\documentclass[11pt]{article}

\usepackage{amsmath,amsthm,amssymb,amsfonts, graphicx}
\usepackage{graphics}
\usepackage{authblk}
\usepackage{upgreek}
\usepackage{amsrefs,hyperref,cleveref}
\usepackage[left=0.5in,right=1in,top=0.7in,bottom=1in]{geometry}
\usepackage{longtable}
\usepackage{epstopdf,tcolorbox}
\usepackage{empheq}
\usepackage{authblk}

\theoremstyle{plain}
\newtheorem{theorem}{Theorem}[section]

\newtheorem{definition}[theorem]{Definition}

\newtheorem{proposition}[theorem]{Proposition}

\theoremstyle{definition}
\newtheorem{remark}[theorem]{Remark}
\numberwithin{equation}{section}
\newcommand{\diff}{\mathop{}\!\mathrm{d}}

\DeclareMathOperator{\tr}{tr}

\DeclareMathOperator{\supp}{supp}

\title{Finite-time blowup for the infinite dimensional vorticity equation}
\author[1]{Evan Miller}
\affil[1]{University of Maine,
evan.miller1@maine.edu}

\begin{document}

\maketitle

\begin{abstract}
    In a previous work with Tai-Peng Tsai \cite{MillerTsai}, the author studied the dynamics of axisymmetric, swirl-free Euler equation in four and higher dimensions. 
    One conclusion of this analysis is that the dynamics become dramatically more singular as the dimension increases.
    In particular, the barriers to finite-time blowup for smooth solutions which exist in three dimensions do not exist in higher dimensions $d\geq 4$.
    Motivated by this result, we will consider a model equation that is obtained by taking the formal limit of the scalar vorticity evolution equation  as $d\to +\infty$. This model exhibits finite-time blowup of a Burgers shock type. The blowup result for the infinite dimensional model equation strongly suggests a mechanism for the finite-time blowup of smooth solutions of the Euler equation in sufficiently high dimensions.
    It is also possible to treat the full Euler equation as a perturbation of the infinite dimensional model equation, although this perturbation is highly singular.
\end{abstract}

\section{Introduction}

The incompressible Euler equation is given by
\begin{align}
    \partial_t u
    +(u\cdot\nabla) u
    +\nabla p
    &=0 \\
    \nabla\cdot u&=0,
\end{align}
where $u\in \mathbb{R}^d$ is the velocity, and $p$ is the pressure, which can be determined from the velocity by solving the Poisson equation
\begin{equation}
-\Delta p= \sum_{i,j=1}^d
\partial_i u_j \partial_j u_i.
\end{equation}
The first equation expresses Newton's second law, $F=ma$: $\partial_t u+ (u\cdot\nabla)u$ is the acceleration of a parcel of fluid, and $-\nabla p$ is the force acting on the particle. The divergence free constraint exists due to the conservation of mass.

Although incompressible Euler equation is among the oldest nonlinear PDEs, core questions about its solutions remain open. In particular, it is an open question whether smooth solutions of the Euler equation can form singularities in three and higher dimensions $d\geq 3$. Kato proved the local existence of strong solutions in $H^s\left(\mathbb{R}^d\right)$ when $s> 1+\frac{d}{2}$, but these solutions are only known to exist locally in time when $d\geq 3$. It is classical that when $d=2$, the Euler equation has global smooth solutions for arbitrary initial data. This is due to the transport of the scalar vorticity $\omega:=\partial_1 u_2-\partial_2 u_1$,
which has the evolution equation
\begin{equation} \label{2dVort}
    \partial_t \omega+(u\cdot\nabla)\omega=0.
\end{equation}
In any dimension, the Beale-Kato-Majda criterion guarantees that the $L^\infty$ norm of $\omega$ controls regularity \cite{BKM}, and in particular if a smooth solution of the Euler equation forms a singularity in finite-time $T_{max}<+\infty$, then
\begin{equation}
    \int_0^{T_{max}}\|\omega(\cdot,t)\|_{L^\infty}
    \diff t
    =+\infty.
\end{equation}
When $d=2$, the transport equation \eqref{2dVort} guarantees that the maximum magnitude of the vorticity cannot grow in time, and so there is a global smooth solution.
Note that when $d=3$, the vorticity is a vector, rather than a scalar, with $\omega=\nabla\times u$. In any dimension, $d\geq 2$ the vorticity is the anti-symmetric part of the velocity gradient, but only when $d=2,3$ does it have a representation as a scalar/vector.

One interesting class of solutions of the Euler equation are axisymmetric, swirl-free solutions.
An axisymmetric, swirl-free vector field can be expressed in the form
\begin{equation}
    u(x)=u_r(r,z)e_r+u_z(r,z)e_z,
\end{equation}
where
\begin{align}
    r&=\sqrt{x_1^2+...+x_{d-1}^2} \\
    z&=x_d \\
    e_r&=\frac{(x_1,...,x_{d-1},0)}{r} \\
    e_z&= e_d \\
    \nabla \cdot u
    &=
    \partial_r u_r
    +\partial_z u_z
    +\frac{d-2}{r}u_r.
\end{align}
Axisymmetry is preserved by the dynamics of the Euler equation. If the initial data $u^0\in H^s\left(\mathbb{R}^d\right)$, where $s>1+\frac{d}{2}$ is axisymmetric and swirl-free, then the unique, strong solution of the Euler equation, $u\in C\left([0,T_{max});H^s\left(\mathbb{R}^d\right)\right)$, will also be axisymmetric and swirl-free for all $0<t<T_{max}$.
For axisymmetric, swirl-free solutions, all of the dynamics can be encoded in the evolution of the scalar vorticity $\omega=\partial_ru_z-\partial_zu_r$, which satisfies the evolution equation
\begin{equation}
    \partial_t\omega+(u\cdot\nabla)\omega
    -k\frac{u_r}{r}\omega=0,
\end{equation}
where $k=d-2$.

Crucially, this evolution equation implies that the quantity $\frac{\omega}{r^k}$ is transported by the flow with
\begin{equation} \label{TransportEqn}
    (\partial_t+u\cdot\nabla)\frac{\omega}{r^k}=0.
\end{equation}
For $d=3$, Ukhovskii and Yudovich first proved global regularity for axisymmetric, swirl-free solutions of the Euler equation when $\frac{\omega^0}{r}$ is bounded in the appropriate space \cite{Yudovich}. This result has seen a number of improvements on the spaces in which $\frac{\omega^0}{r}$ must be bounded---see \cite{Danchin} and the references therein.
In particular, we will note that the criterion in \cite{Danchin} that $\omega^0\in L^\infty$ and $\frac{\omega^0}{r}\in L^{3,1}$ holds automatically for any axisymmetric, swirl-free initial data $u^0\in H^s\left(\mathbb{R}^3\right), s>\frac{5}{2}$, which is the natural Hilbert space for strong wellposedness, going back to the foundational work of Kato \cite{KatoLocalWP}. 

We will note that the requirement that $\omega$ vanish linearly at the axis $r=0$ is a regularity condition, not a boundary condition, although it is expressed as a boundary condition in coordinates. For axisymmetric, swirl-free vector fields, the anti-symmetric part of the velocity gradient,
\begin{equation}
    A_{ij}=\frac{1}{2}\left(\partial_i u_j -\partial_j u_i\right),
\end{equation}
can be expressed in terms of the the vorticity by
\begin{equation}
    A(x)=\frac{1}{2}\omega(r,z) 
    \left(e_r\otimes e_z
    -e_z\otimes e_r\right).
\end{equation}
The requirement that $\omega$ vanish linearly at the axis comes from the fact that $e_r$ has a singularity at the axis, while $re_r$ does not. If $\omega$ does not vanish at the axis, then $A$ will not be smooth.

In a major recent breakthrough, Elgindi proved finite-time blowup for $C^{1,\alpha}\left(\mathbb{R}^3\right)$ solutions of the Euler equation \cite{Elgindi}. These solutions are axisymmetric and swirl-free, but only vanish at a rate $\omega \sim r^\alpha$ for $\alpha \ll 1$. It is a key aspect of the proof that alpha is small, which enables a perturbative regime. Because $\frac{\omega^0}{r}\notin L^{3,1}$, this does not violate previous results for global existence of axisymmetric, swirl-free solutions of the Euler equation. This has direct bearing on the possible blowup of smooth solutions in four and higher dimensions. While $\frac{\omega^0}{r}$ is bounded in any dimension for smooth initial data, $\frac{\omega^0}{r^k}$ can be unbounded even for smooth initial data when $d\geq 4$---see \cite{MillerTsai} for a derivation of the transport equation \eqref{TransportEqn} and for examples/discussion. If the transported quantity $\frac{\omega}{r}$ being unbounded can lead to blowup for $C^{1,\alpha}$ solutions when $d=3$, then the transported quantity $\frac{\omega}{r^k}$ being unbounded could lead to finite-time blowup of smooth solutions when $d\geq 4$, and in particular in very high dimensions.

To justify this line of study, we will consider the dynamics of a model equation for the vorticity equation in infinite-dimensional limit.
In three and higher dimensions, all of the dynamics of axisymmetric, swirl-free solutions of the Euler equation can be expressed in terms of the vorticity evolution equation and a stream function as follows:
\begin{align} \label{VortStreamEqn}
    \partial_t \omega +(u\cdot\nabla)\omega
    -k\frac{u_r}{r}\omega
    &=0 \\
    u_r&= \partial_z \psi \\
    u_z&=-\partial_r\psi-k\frac{\psi}{r} \\
    \left(-\partial_z^2-\partial_r^2
    -\frac{k}{r}\partial_r
    +\frac{k}{r^2}\right)\psi 
    &=\omega \label{StreamPDE}.
\end{align}
    The Poisson-type equation in \eqref{StreamPDE} allows the stream function to be expressed in terms of a singular integral operator applied to the vorticity as follows:
    \begin{equation} \label{PoissonKernel}
    \psi(r,z)=
    -\frac 1{2\pi}\int_{-\infty}^\infty \int_0^\infty \int_0^\pi \frac {(\sin \theta)^{k+1} r^{d-1}\bar{r}^{d-1}}
{\left((r-\bar{r})^2+2r\bar{r}(1-\cos(\theta))+(z-\bar{z})^2\right)^{\frac{d}2}} \omega(\bar{r},\bar {z})
\diff\theta \diff\bar{r}\diff\bar{z}.
    \end{equation}
This is classical in three dimensions; for the derivation for four and higher dimensions,
see Appendix A in \cite{MillerTsai} for details.

It is possible to take the formal limit of this system as $d\to \infty$, which yields the infinite-dimensional vorticity equation
\begin{align} 
    \partial_t\omega+\phi\partial_z\omega
    +\omega\partial_z \phi&=0 \\
    \partial_r \phi&=\omega. 
\end{align}
We will prove that smooth solutions of the infinite-dimensional vorticity equation can exhibit finite-time blowup of a Burgers shock type.

\begin{theorem} \label{InfiniteVortThmIntro}
Suppose $\omega^0\in C_c^\infty\left(\mathbb{R}^+\times\mathbb{R}\right)$
and $\omega^0=\partial_r\phi^0$.
Then there exists a unique strong solution 
$\omega \in C^\infty\left([0,T_{max});
C_c^\infty\left(\mathbb{R}^+
\times\mathbb{R}\right)\right)$ to the infinite-dimensional vorticity equation.
If $\omega^0$ is not identically zero, then there is finite-time blowup with
\begin{equation}
    T_{max}=\frac{1}{-\inf_{\substack{r\in\mathbb{R}^+\\ z\in\mathbb{R}}}\partial_z\phi^0(r,z)}.
\end{equation}
This solution of the infinite-dimensional vorticity equation is given by
\begin{equation}
    \omega(r,z,t)=\frac{\omega^0(r,h(r,z,t))}
    {1+t\partial_y\phi^0(r,h(r,z,t))},
\end{equation}
where
\begin{equation}
\phi^0(r,z)
=
\partial_r^{-1}\omega^0(r,z)
:=
-\int_{r}^\infty \omega^0(\rho,z)\diff\rho
\end{equation}
\begin{equation}
    h(r,z,t)=g_{r,t}^{-1}(z),
\end{equation}
is the back-to-labels map of the flow given by
\begin{equation}
    g_{r,t}(z)=z+\phi^0(r,z)t.
\end{equation}
\end{theorem}

The fact that the infinite-dimensional vorticity equation exhibits finite-time blowup further suggests the possibility that in sufficiently high dimension, axisymmetric, swirl-free solutions of the Euler equation blowup in finite-time. This expands on previous arguments by the author and Tsai in \cite{MillerTsai}, which established a conditional blowup result for axisymmetric, swirl-free solutions of the Euler equation when $d\geq 4$, and furthermore the condition becomes weaker as $d\to +\infty$. In particular, we prove finite-time blowup for solutions with a colldiding vortex ring geometry---that is $\omega(r,z)$ is odd in $z$ and nonnegative for $z>0$--- when the proportion of energy in the $z$ direction is bounded below uniformly in time with
\begin{equation}
\frac{K_z}{K_0}>\frac{1+\delta}{d}.
\end{equation}
for some $\delta>0$. For details, see Theorem 1.6 in \cite{MillerTsai}. Furthermore, Hou and Zhang \cite{HouZhang} provided numerical evidence for axisymmetric, swirl-free solutions of the Euler equation $C^{1,\alpha}$, 
whenever $\alpha<1-\frac{2}{d}$. While this would not be a smooth blowup even if it could be established rigourously, it is still an improvement on the best known results for finite-time blowup of the Euler equation in very large dimension, as the arguments in \cite{Elgindi} require that the vorticity be $C^{1,\alpha}$
where $\alpha\ll 1$ is a perturbative parameter, and is therefore necessarily small.

In \cref{PerturbSection}, we will even discuss a way to write the vorticity equation in very high dimensions as a perturbation of the infinite-dimensional vorticity equation, where the perturbation has size $\epsilon=\frac{1}{d-2}$.
This could potentially allow blowup to be proven perturbatively for large enough $d\gg 3$, however the limit of the perturbation as $\epsilon \to 0$ is extremely singular, so any perturbative argument would need to be quite subtle in order to overcome the difficulties posed by this singular limit.

\begin{theorem} \label{PerturbThmIntro}
Suppose $u\in C\left([0,T_{max});H^s
\left(\mathbb{R}^d\right)\right)$ is an axisymmetric, swirl-free solution of the Euler equation in $d$-dimensional space, with $s>\frac{d}{2}+2, d\geq 3$.
Then the vorticity satisfies the equation
\begin{equation}
    \partial_t\omega 
    +\phi\partial_z\omega
    +\omega\partial_z\phi
    +\epsilon Q_\epsilon[\omega]=0,
\end{equation}
where
\begin{align}
    \phi=\partial_r^{-1}\omega \\
    \epsilon=\frac{1}{d-2},
\end{align}
and the operator $Q_\epsilon$ is defined by
\begin{equation}
    Q_\epsilon[\omega]=
    (-r\partial_z\phi
    +\epsilon\partial_z\sigma)
    \partial_r\omega
    +\left(\phi+r\omega
    -\frac{\sigma}{r}
    -\epsilon\partial_r\sigma
    \right)\partial_z\omega
    -\frac{\partial_z\sigma}{r}\omega.
\end{equation}
We define $\sigma$ to be the unique solution of the Poisson-type equation
\begin{equation}
    -\Delta_\epsilon \sigma
    =
    -\left(\partial_r^2+\partial_z^2\right)
    (r\phi),
\end{equation}
with decay as $r,|z|\to \infty$,
where
\begin{equation}
    -\Delta_\epsilon
    =
    -\epsilon\left(\partial_r^2
    +\partial_z^2\right)
    -\frac{1}{r}\partial_r
    +\frac{1}{r^2}.
\end{equation}
Note that $\sigma$ can be expressed as an integral against a Poisson-type kernel as in \eqref{PoissonKernel}.
\end{theorem}

\begin{remark}
    \Cref{InfiniteVortThmIntro} relates to a previous result of Drivas and Elgindi on the dependence of the Euler equation on the dimension \cite{DrivasElgindi}. In this paper, they consider flows where $\nabla u$ is nilpotent, considering in particular solutions on $\mathbb{T^d}$ $u_i$ only depends on $x_{i+1},...x_d$.
    For solutions of this form
    \begin{equation}
        \nabla\cdot (u\cdot\nabla)u=0,
    \end{equation}
    and so these are pressure-less solutions of the Euler equation and also satisfy the $d$-dimensional Burgers equation. While these solutions cannot blowup in finite-time, in the limit as $d\to +\infty$, they come closer to gradient blowup at the rate $|\nabla u(x,t)|\sim \frac{1}{C-t}$ known classically. In this case, however, there is no infinite-dimensional limiting equation, even as a purely formal limit. Likewise, axisymmetric swirl-free solutions of the Euler equation could potentially exhibit finite-time blowup in sufficiently large dimension, while blowup in any finite-dimension is ruled out for the Ansatz considered in \cite{DrivasElgindi}.

    Drivas and Elgindi note that the divergence free constraint become less restrictive in higher dimensions, meaning that pressure plays less of a role \cite{DrivasElgindi}, which could allow the formation of a Burgers type shock. We observe this exact fact in the infinite-dimensional Euler equation. The divergence free constraint for axisymmetric, swirl-free vector fields in $\mathbb{R}^d$ can be expressed as
    \begin{equation}
    \frac{1}{d-2}(\partial_ru_r
    +\partial_zu_z)+\frac{u_r}{r}=0.
    \end{equation}
    The formal limit of this equation is just $u_r=0$. This leaves the evolution equation
    \begin{equation}
    \partial_t u_z+u_z\partial_z u_z=0,
    \end{equation}
    which satisfies the ``infinite dimensional'' divergence free constraint $u_r=0$ without requiring a pressure term. The Burgers type blowup in \Cref{InfiniteVortThmIntro} is then unsurprising. The difficulty is getting this kind of singularity formation for large, but finite, dimension, given the difficulty of perturbing from infinity. \Cref{PerturbThmIntro} suggests an approach to this problem, but the $\epsilon\to 0$ limit of this problem is so singular, that the analysis would have to be very delicate.
\end{remark}

\section{The infinite dimensional limit of the vorticity equation} \label{InfiniteVortSection}

In this section, we will derive a formal limit for the vorticity equation \eqref{VortStreamEqn} as $d\to \infty$. 
Note that the dimension $d=k+2$ enters these equations as a parameter;
the qualitative structure of the equations is the same for any $d\geq 3$.
Therefore, it is possible to study how the dynamics of this equation change as this parameter changes, and in particular to consider the infinite dimensional limit
$k\to+\infty$.

Because $d=k+2$, taking the infinite dimensional limit is equivalent to taking the formal limit $k\to \infty$ in the stream function formulation. If we take the stream function formulation from the end of the last section, there are terms of order $k$, which makes taking the limit $k\to\infty$ difficult.

We will deal with this issue by taking the re-normalization
\begin{equation}
    \Tilde{\psi}=k\psi.
\end{equation}
The stream function formulation of the vorticity equation in terms of $\Tilde{\psi}$ can then be given as follows:
the velocity is determined in terms of $\Tilde{\psi}$
\begin{align}
    u_z&=
    -\frac{1}{k}\partial_r\Tilde{\psi}
    -\frac{\Tilde{\psi}}{r} \\
    u_r&= 
    \frac{1}{k}\partial_z \Tilde{\psi},
\end{align}
and $\Tilde{\psi}$ is determined by the vorticity from the elliptic equation
\begin{equation}
    \left(-\frac{1}{k}\partial_z^2
    -\frac{1}{k}\partial_r^2
    -\frac{1}{r}\partial_r
    +\frac{1}{r^2}\right)\Tilde{\psi}
    =\omega.
\end{equation}
The evolution equation remains
\begin{equation}
    \partial_t\omega
    +u_r\partial_r\omega
    +u_z\partial_z\omega 
    -\frac{k}{r}u_r\omega
    =0.
\end{equation}
Plugging in the velocity given by $\Tilde{\psi}$ to this evolution equation, we find that
\begin{equation}
    \partial_t\omega-\left(
    \frac{1}{k}\partial_r\Tilde{\psi}
    +\frac{\Tilde{\psi}}{r}
    \right)\partial_z\omega 
    +\frac{1}{k}\partial_z\Tilde{\psi} \partial_r\omega
    -\frac{\partial_z\Tilde{\psi}}{r}
    \omega=0.
\end{equation}

Taking the formal limit $k\to +\infty,$
we obtain the equations
\begin{align}
    \omega &=
    -\frac{1}{r}\partial_r\Tilde{\psi}
    +\frac{1}{r^2}\Tilde{\psi} \\
    &=
    -\partial_r\left(
    \frac{\Tilde{\psi}}{r}\right),
\end{align}
and
\begin{equation}
   \partial_t\omega
   -\frac{\Tilde{\psi}}{r}
   \partial_z\omega 
   -\omega\partial_z
   \left(\frac{\Tilde{\psi}}{r}\right)
   =0.
\end{equation}
If we make the substitution
\begin{equation}
    \phi=-\frac{\Tilde{\psi}}{r},
\end{equation}
then our equation reduces to
\begin{align}
    \partial_t\omega +\phi \partial_z \omega
    +\omega \partial_z \phi&=0,
    \label{VortInfinite} \\
    \partial_r\phi &=\omega. \label{VortPotential}
\end{align}

Taking the boundary condition 
\begin{equation}
    \lim_{r\to +\infty} \phi(r,z)=0,
\end{equation}
for all $z\in\mathbb{R}$, and integrating the equation $\eqref{VortPotential}$ from infinity, we find that
\begin{equation}
    \phi(r,z)=-\int_r^{+\infty} \omega(\rho,z)\diff\rho
\end{equation}

\begin{definition} \label{InverseDef}
Define the operator $\partial_r^{-1}:
C_c^\infty\left(\mathbb{R}^+\times \mathbb{R}\right) \to
C_c^\infty\left(\mathbb{R}^+\times \mathbb{R}\right)$ by
\begin{equation}
    \left(\partial_r^{-1}\omega\right)(r,z)
    =-\int_r^{+\infty} \omega(\rho,z)\diff\rho.
\end{equation}
Note that we take $\mathbb{R}^+=[0,+\infty)$, and so $\omega\in C_c^\infty\left(\mathbb{R}^+\times \mathbb{R}\right)$ does not require that $\omega(0,z)=0$. We only require that there exist $R,Z>0$ such that 
\begin{equation}
    \supp(\omega)\subset [0,R]\times [-Z,Z].
\end{equation}
It is straightforward to observe that if $r>R$ or $|z|>Z$, then
\begin{equation}
    \left(\partial_r^{-1}\omega\right)(r,z)
    =-\int_r^{+\infty} \omega(\rho,z)\diff\rho
    =0,
\end{equation}
and so $\partial_r^{-1}\omega\in C_c^\infty\left(\mathbb{R}^+\times \mathbb{R}\right)$, and the map is well defined.
\end{definition}

Note that $\phi=\partial_r^{-1}\omega$
is the unique solution of the equation $\partial_r\phi=\omega$ with compact support.

\begin{proposition} \label{InverseProp}
    The operator $\partial_r: C_c^\infty\left(\mathbb{R}^+\times \mathbb{R}\right) \to
C_c^\infty\left(\mathbb{R}^+\times \mathbb{R}\right)$ is a bijection with inverse $\partial_r^{-1}$.
\end{proposition}

\begin{proof}
    Suppose $\phi\in C_c^\infty\left(\mathbb{R}^+\times \mathbb{R}\right)$ and $\partial_r\phi=0$. Then the mean value theorem and the compact support condition imply $\phi=0$. 
    Because $\partial_r$ is a linear operator, this means that $\partial_r: C_c^\infty\left(\mathbb{R}^+\times \mathbb{R}\right) \to
C_c^\infty\left(\mathbb{R}^+\times \mathbb{R}\right)$ is injective, and therefore has an inverse. The fact that this inverse $\partial_r^{-1}$ is given in \Cref{InverseDef} follows immediately from the fundamental theorem of calculus, and this formula establishes that the map $\partial_r$ is surjective, and therefore a bijection, because for all $\omega\in C_c^\infty\left(\mathbb{R}^+\times \mathbb{R}\right)$, 
\begin{equation}
    \partial_r\left(\partial_r^{-1}\omega\right)
    =\omega.
\end{equation}
\end{proof}

\begin{definition}
    We will say that $\omega \in 
    C^\infty\left([0,T_{max});
C_c^\infty\left(\mathbb{R}^+
\times\mathbb{R}\right)\right)$ is a strong solution of the infinite dimensional vorticity equation if
\begin{equation}
    \partial_t\omega +
    \phi\partial_z\omega
    +\omega\partial_z\phi
    =0,
\end{equation}
where $\phi=\partial_r^{-1}\omega$
\end{definition}

\begin{remark}
    Note that our definition of $C_c^\infty\left(\mathbb{R}^+
\times\mathbb{R}\right)$ does not require the solution to be supported away from the axis $r=0$, and so there is no requirement that $\omega(0,z)=0$ or $\phi(0,z)=0$ for our wellposedness theory. For the full axisymmetric, swirl-free Euler equation in $\mathbb{R}^d$, for any dimension $d\geq 3$, we require that $\omega(0,z)=0$ in order for a solution to be smooth. More particularly, if $u\in C^2\left(\mathbb{R}^d\right)$ is axisymmetric and swirl free, then the vorticity must vanish linearly at the axis with $\frac{\omega}{r}\in L^\infty$. It is reasonable then, to consider data where $\omega(0,z)=0$ for all $z\in\mathbb{R}$, and we can see that this condition is preserved by the dynamics. The quantity $\phi$ in the infinite dimensional case corresponds to $u_z$ in the finite-dimensional case, and so, unlike with the vorticity, there is no geometric reason to impose the condition $\phi(0,z)=0$.

Note furthermore that this definition of $C_c^\infty\left(\mathbb{R}^+
\times\mathbb{R}\right)$ is natural for this problem because an axisymmetric swirl-free vector field is in $C_c^\infty\left(\mathbb{R}^d\right)$ if and only if its components expressed in cylindrical coordinates are in $C_c^\infty\left(\mathbb{R}^+
\times\mathbb{R}\right)$ as we have defined it.
\end{remark}

\begin{remark}
We will note that the infinite dimensional vorticity equation \eqref{VortInfinite} can, shockingly enough, be further simplified to the one dimensional Burgers equation. Recalling that $\omega=\partial_r\phi$, we can see that \eqref{VortInfinite} can be rewritten as
\begin{equation}
    \partial_t\partial_r\phi
    +\phi\partial_z\partial_r\phi
    +\partial_r\phi\partial_z\phi
    =0,
\end{equation}
which can be in turn expressed as
\begin{equation}
    \partial_t\partial_r \phi
    +\partial_r\left(\phi\partial_z\phi\right)=0.
\end{equation}
Applying the operator $\partial_r^{-1}$, we can see that $\phi$ must satisfy Burgers equation,
\begin{equation}
    \partial_t\phi+\phi\partial_z\phi=0.
\end{equation}
This formally establishes the reduction from the infinite dimensional vorticity equation to the one dimensional Burgers equation. We will give a more rigourous proof shortly.
\end{remark}

\section{Burgers shock type blowup for the infinite dimensional vorticity equation}

In this section, we will consider finite-time blowup for the infinite-dimensional vorticity equation derived in the previous section.
We will begin by formally proving the equivalence of solutions of the infinite dimensional vorticity equation and Burgers equation under the map $\partial_r^{-1}$.

\begin{proposition} \label{BurgersVortEquiv}
Suppose $\phi \in
C^\infty\left([0,T_{max});
C_c^\infty\left(\mathbb{R}^+
\times\mathbb{R}\right)\right)$ is a strong solution of Burgers equation, and let $\omega=\partial_r\phi$. Then
$\omega \in C^\infty\left([0,T_{max});
C_c^\infty\left(\mathbb{R}^+
\times\mathbb{R}\right)\right)$
is a strong solution of the infinite dimensional vorticity equation.
\end{proposition}

\begin{proof}
We know that $\phi$ is a strong solution to Burger's equation, and so for all
$(r,z,t)\in
\mathbb{R}^+\times\mathbb{R}
\times[0,T_{max})$,
\begin{equation}
    \partial_t\phi +\phi\partial_z\phi=0.
\end{equation}
Differentiating this equation with respect to $r$, we find that for all
$(r,z,t)\in\left(
\mathbb{R}^+,\mathbb{R},[0,T_{max}\right)$,
\begin{align}
    \partial_r\partial_t\phi
    +\partial_r(\phi\partial_z\phi)&=0 \\
    \partial_t\partial_r\phi 
    +\phi\partial_z\partial_r\phi
    +\partial_r\phi\partial_z\phi
    &=0 \\
    \partial_t\omega+\phi\partial_z\omega
    +\omega\partial_z\phi &=0.
\end{align}
This completes the proof.
\end{proof}

\begin{remark}
The requirement that $\phi \in C_c^\infty$ is stronger than is strictly necessary. We really only need the equality of the mixed partials, $\partial_r\partial_t\phi
=\partial_t\partial_r\phi$,
and 
$\partial_r\partial_z\phi
=\partial_z\partial_r\phi$.
Requiring that $\phi$ be smooth is a convenient way to guarantee the equality of the mixed partials will hold, and we are interested in the blowup of smooth solutions with decay at infinity, and so $C_c^\infty$ is a convenient class to work with.
\end{remark}

\begin{proposition} \label{VortBurgersEquiv}
Suppose $\omega \in C^\infty\left([0,T_{max});
C_c^\infty\left(\mathbb{R}^+
\times\mathbb{R}\right)\right)$
is a strong solution of the infinite dimensional vorticity equation,
and $\phi=\partial_r^{-1}\omega$.
Then
$\phi \in C^\infty\left([0,T_{max});
C_c^\infty\left(\mathbb{R}^+
\times\mathbb{R}\right)\right)$
is a strong solution of Burgers equation.
\end{proposition}

\begin{proof}
Applying \Cref{InverseProp} we can see that 
$\phi \in C^\infty\left([0,T_{max});
C_c^\infty\left(\mathbb{R}^+
\times\mathbb{R}\right)\right)$
and $\omega=\partial_r\phi$.
By hypothesis, we have
\begin{equation}
    \partial_t\omega
    +\phi\partial_z\omega
    +\omega\partial_z\phi=0.
\end{equation}
Substituting in $\omega=\partial_r\phi$, and using the fact that $\phi\in C^2$,
we can see that for all 
$(r,z,t)\in\left(
\mathbb{R}^+,\mathbb{R},[0,T_{max})\right)$,
\begin{align}
    \partial_r\left(\partial_t\phi
    +\phi\partial_z\phi\right)
    &=
    \partial_t\partial_r\phi
    +\phi\partial_z\partial_r\phi
    +\partial_r\phi\partial_z\phi\\
    &=0.
\end{align}
We showed in \Cref{InverseProp} that $\partial_r$ is a bijection on $C_c^\infty\left(\mathbb{R}^+\times \mathbb{R}\right)$,
so this implies that 
for all $(r,z,t)\in\left(
\mathbb{R}^+,\mathbb{R},[0,T_{max})\right)$,
\begin{equation}
    \partial_t\phi
    +\phi\partial_z\phi=0,
\end{equation}
which completes the proof.
\end{proof}

Now that we have shown that the infinite-dimensional vorticity equation and the 1D Burgers equation are equivalent under the map $\partial_r^{-1}$,
we will recall the standard theory for the existence, uniqueness, and finite-time blowup of solutions of Burgers equation, with proofs provided for the sake of completeness, as the flow map from the proof will be important in discussing the solutions to the infinite dimensional vorticity equation that are generated by solutions of Burgers equation.

\begin{theorem} \label{BurgersThm}
Suppose $\phi^0\in 
C_c^\infty\left(\mathbb{R}^+
\times\mathbb{R}\right)$. Then there exists a unique, strong solution of Burgers equation
$\phi\in C^\infty\left([0,T_{max});
C_c^\infty\left(\mathbb{R}^+
\times\mathbb{R}\right)\right)$.
If $\partial_z\phi^0(r,z)\geq 0$ for all $r\in\mathbb{R}^+,z\in \mathbb{R},$
then there is a global smooth solution, and so $T_{max}=+\infty$. If there exists $r_0\in\mathbb{R}^+z_0\in\mathbb{R}$, such that $\partial_z\phi^0(r_0,z_0)<0$, then there is finite-time blowup with
\begin{equation}
    T_{max}=\frac{1}{-\inf_{\substack{r\in\mathbb{R}^+\\ z\in\mathbb{R}}}\partial_z\phi^0(r,z)}.
\end{equation}
This solution is given by
\begin{equation} \label{BurgersKernel}
    \phi(r,z,t)=
    \phi^0(r,h(r,z,t),t),
\end{equation}
where
\begin{equation}
    h(r,z,t)=g_{r,t}^{-1}(z),
\end{equation}
and
\begin{equation}
    g_{r,t}(z)=z+\phi^0(r,z)t.
\end{equation}
\end{theorem}

\begin{proof}
The Burgers equation involves the advection of $\phi$ by itself, resulting in the flow map
\begin{equation}
    z \to z+t\phi^0(r,z)
\end{equation}
and the back to labels map $h(r,z,t)$.
We will begin by computing the derivatives of $h$ and introducing the variable $y=h(r,z,t)$.

First observe that $g_{r,t}(z)$ is smooth and that 
\begin{equation}
    \partial_zg_{r,t}(z)=
    1+t\partial_z\phi^0(r,z).
\end{equation}
This means that if for all $r\in\mathbb{R}^+,z\in\mathbb{R}$,
\begin{equation}
    \partial_z \phi^0(r,z)\geq0,
\end{equation}
then for all $t>0$,
\begin{equation}
    \partial_z g_{r,t}(z)\geq 1.
\end{equation}
Furthermore, if this derivative is negative at some point, then for all $0<t<T_{max}$,
\begin{equation}
    \partial_z g_{r,t}(z) \geq 
    1-\frac{t}{T_{max}}.
\end{equation}
In either case, for all $0<t<T_{max}$, we can conclude that $g_{r,t}(z)$ is smooth and strictly increasing in $z$ with a derivative bounded below by a positive constant. Consequently it has an inverse that is also smooth, and therefore
\begin{equation}
    \phi(r,z,t)=\phi^0(r,h(r,z,t))
\end{equation}
is a well defined function 
$\phi\in C^\infty\left([0,T_{max});
C_c^\infty\left(\mathbb{R}^+
\times\mathbb{R}\right)\right)$.

It remains to show that $\phi$ is in fact a solution of Burgers equation.
Observe that by definition,
\begin{align}
    z&=g_{r,t}(h(r,z,t)) \\
    &=g_{r,t}(y) \\
    &=y+t\phi^0(r,y),
\end{align}
and also
\begin{equation} \label{hDef}
    h(r,y+t\phi^0(r,y),t)=y.
\end{equation}
Differentiating \eqref{hDef} with respect to $y$, we find that
\begin{equation}
    \left(1+t\partial_y\phi^0(r,y)\right)
    \partial_zh(r,z,t)=1.
\end{equation}
Plugging for $y$ and dividing over a term, we find that
\begin{equation}
    \partial_z h(r,z,t)=\frac{1}
    {1+t\partial_y\phi^0(r,h(r,z,t))}
\end{equation}
Meanwhile, differentiating \eqref{hDef} with respect to $t$, we find that
\begin{equation}
    \partial_t h(r,z,t)
    +\phi^0(r,y)\partial_z h(r,z,t)=0.
\end{equation}
Rearranging terms, we find that 
\begin{equation}
    \partial_t h(r,z,t)
    =\frac{-\phi^0(r,h(r,z,t))}
    {1+t\partial_y\phi^0(r,h(r,z,t))}.
\end{equation}

Now we can show that
\begin{equation}
    \phi(r,z,t)=\phi^0(r,h(r,z,t))
\end{equation}
is a solution to Burgers equation.
First we compute that
\begin{align}
    \partial_t\phi(r,z,t)
    &=
    \partial_y\phi^0(r,h(r,z,t))
    \partial_t h(r,z,t) \\
    &=
    \frac{-\phi^0(r,y)
    \partial_y\phi^0(r,y)}
    {1+t\partial_y\phi^0(r,y)}.
\end{align}
Likewise we compute that
\begin{align}
    \partial_z\phi(r,z,t)
    &=
    \partial_y\phi^0(r,h(r,z,t))
    \partial_z h(r,z,t) \\
    &=
    \frac{\partial_y\phi^0(r,y)}
    {1+t\partial_y\phi^0(r,y)}.
\end{align}
Putting these computations together with the fact that $\phi(r,z,t)=\phi^0(r,y)$, we can conclude that
\begin{equation}
    \partial_t\phi(r,z,t)
    +\phi(r,z,t)\partial_z\phi(r,z,t)=0,
\end{equation}
and so we have a solution to Burgers equation.

We have now proven that there is a strong solution of Burgers equation on the interval $[0,T_{max})$. We still need to show that there is in fact blowup as $t\to T_{max}$ when $T_{max}<+\infty.$
The key point is that when there exists $r_0\in\mathbb{R}^+, z_0 \in\mathbb{R}$ such that $\partial_z\phi(r_0,z_0)<0$, then for $t>T_{max}$ the map
\begin{equation}
    z \to z+t\phi^0(r,z)
\end{equation}
is no longer invertible. The derivative with respect to $z$ of the inverse of this map becomes singular as $t\to T_{max}$, which produces the Burgers shock.
To proceed rigourously, recall that 
\begin{equation}
    \partial_z\phi(r,z,t)
    =
    \frac{\partial_y\phi^0(r,y)}
    {1+t\partial_y\phi^0(r,y)}.
\end{equation}
This clearly implies that
\begin{equation}
    \lim_{t\to T_{max}}\inf\partial_z\phi(r,z,t)
    =-\infty,
\end{equation}
as the denominator is clearly going to $0$ as
$t\to \frac{-1}{\inf\partial_y\phi^0(r,y)}$.
This completes the proof of finite-time blowup in the from of the Burgers shock.
Note that uniqueness for classical solutions follows immediately from the characteristics we have derived, because any strong solution must follow these characteristics.
\end{proof}

With the classical well-posedness theory of the one dimensional Burgers equation worked out, we can now immediately lift this theory to the infinite-dimensional vorticity equation.
We will now prove Theorem \ref{InfiniteVortThmIntro}, which is restated for the readers convenience.

\begin{theorem} \label{InfiniteVortThm}
Suppose $\omega^0\in 
C_c^\infty\left(\mathbb{R}^+
\times\mathbb{R}\right)$,
and let $\phi^0=\partial_r^{-1}\omega^0$.
Then there exists a unique strong solution 
$\omega \in C^\infty\left([0,T_{max});
C_c^\infty\left(\mathbb{R}^+
\times\mathbb{R}\right)\right)$ to the infinite-dimensional vorticity equation.
If $\partial_z\phi^0(r,z)\geq 0$ for all $r\in\mathbb{R}^+,z\in \mathbb{R},$
then there is a global smooth solution, and so $T_{max}=+\infty$. If there exists $r_0\in\mathbb{R}^+z_0\in\mathbb{R}$, such that $\partial_z\phi^0(r_0,z_0)<0$, then there is finite-time blowup with
\begin{equation}
    T_{max}=\frac{1}{-\inf_{\substack{r\in\mathbb{R}^+\\ z\in\mathbb{R}}}\partial_z\phi^0(r,z)}.
\end{equation}
This solution of the infinite-dimensional vorticity equation is given by
\begin{equation} \label{VortEq}
    \omega(r,z,t)=\frac{\omega^0(r,h(r,z,t))}
    {1+t\partial_y\phi^0(r,h(r,z,t))},
\end{equation}
where the back-to-labels map $h(r,z,t)$ is defined as in Theorem \ref{BurgersThm}.
\end{theorem}

\begin{proof}
Applying Propositions \ref{BurgersVortEquiv} and \ref{VortBurgersEquiv} to the solutions of Burgers equation in Theorem \ref{BurgersThm}, the local existence and uniqueness of strong solutions, as well as the specified blowup time follows immediately. It only remains to show that the equation \eqref{VortEq} holds for this unique solution.

Recall that 
\begin{equation}
    \phi(r,z,t)=\phi^0(r,h(r,z,t)),
\end{equation}
and that $\omega=\partial_r\phi$.
Applying the chain rule we find that
\begin{equation} \label{Step1}
    \omega(r,z,t)
    =
    \partial_r\phi^0(r,y)
    +\partial_y\phi^0(r,y)\partial_r h(r,z,t).
\end{equation}
Also recall that
\begin{equation}
    h(r,y+t\phi^0(r,y),t)=y,
\end{equation}
and that therefore differentiating with respect to $r$, we find that
\begin{equation}
    \partial_rh(r,z,t)+\partial_zh(r,z,t)t
    \partial_r\phi^0(r,y)=0.
\end{equation}
Subtracting over the latter term and plugging into our expression for $\partial_z h$,
\begin{equation}
    \partial_z h(r,z,t)=\frac{1}
    {1+t\partial_y\phi^0(r,y)},
\end{equation}
we find that
\begin{equation} \label{Step2}
    \partial_rh(r,z,t)=
    \frac{-t\omega^0(r,y)}
    {1+t\partial_y\phi^0(r,y)}.
\end{equation}
Plugging \eqref{Step2} back into our equation \eqref{Step1}, we find that for all 
$0\leq t<T_{max}$,
\begin{align}
    \omega(r,z,t)
    &=\omega^0(r,y)
    +\frac{-t\partial_y\phi^0(r,y)\omega^0(r,y)}
    {1+t\partial_y\phi^0(r,y)} \\
    &=
    \frac{\omega^0(r,y)}
    {1+t\partial_y\phi^0(r,y)} \\
    &=
    \frac{\omega^0(r,h(r,z,t))}
    {1+t\partial_y\phi^0(r,h(r,z,t))} .
\end{align}
This completes the proof.
\end{proof}

In \Cref{InfiniteVortThm} we proved finite-time blowup for all initial data satisfying
\begin{equation}
    \inf_{\substack{r\in\mathbb{R}^+\\ z\in\mathbb{R}}}\partial_z\phi^0(r,z)<0,
\end{equation} 
whereas in the introduction the result was stated for all non-trivial initial data. We will now show that this condition is satisfied for any data that is not identically zero.

\begin{proposition}
    Suppose $\omega \in 
C_c^\infty\left(\mathbb{R}^+
\times\mathbb{R}\right)$ is not identically zero. Then
\begin{equation}
    \inf_{\substack{r\in\mathbb{R}^+\\ z\in\mathbb{R}}}\partial_z\phi(r,z)<0.
\end{equation} 

\begin{proof}
    Because $\partial_r^{-1}$ is a bijective---see \Cref{InverseProp}, we can conclude that $\phi$ is not identically zero, and therefore there exists $(r_0,z_0)
    \in\mathbb{R}^+\times\mathbb{R})$ such that $\phi(r_0,z_0)\neq 0$.
    Suppose without loss of generality $\phi(r_0,z_0)> 0$. The compact support condition implies that there exists $Z>z_0$ such that $\phi(r_0,Z)=0$. The mean value theorem then implies that there exists $z'\in(z_0,Z)$ such that 
    \begin{equation}
    \partial_z \phi (r_0,z')
    =-\frac{\phi(r_0,z_0)}{Z-z_0}
    <0,
    \end{equation}
    and so
    \begin{equation}
    \inf_{\substack{r\in\mathbb{R}^+\\ z\in\mathbb{R}}}\partial_z\phi(r,z)<0.
\end{equation} 
The case where $\phi(r_0,z_0)<0$ is entirely analogous, simply taking $Z<z_0$,
such that $\phi(r_0,Z)=0$.
\end{proof}

\end{proposition}

A variant of the Beal-Kato-Majda criterion does apply to the infinite-dimensional vorticity equation, but only in the case where $\omega^0(r_0,z_0)\neq 0$ at the point $(r_0,z_0)$ where the infimum of $\partial_z\phi$ is attained. In the case where $\omega^0(r_0,z_0)=0$ at this point, then the Beale-Kato-Majda criterion may fail for the infinite-dimensional vorticity equation.

\begin{proposition} \label{BKM}
Suppose $\omega \in C^\infty\left([0,T_{max});
C_c^\infty\left(\mathbb{R}^+
\times\mathbb{R}\right)\right)$ is a strong solution to the infinite-dimensional vorticity equation with finite-time blowup at $T_{max}<+\infty$.
Suppose that
\begin{equation}
    \inf_{\substack{r\in\mathbb{R}^+\\ z\in\mathbb{R}}}\partial_z\phi(r,z)
    =\partial_z\phi^0(r_0,z_0),
\end{equation}
and that
\begin{equation}
    \omega^0(r_0,z_0)\neq 0,
\end{equation}
then
\begin{equation}
    \int_0^{T_{max}}
    \|\omega(\cdot,t)\|_{L^\infty} \diff t
    =+\infty.
\end{equation}
In particular, for all $0\leq t<T_{max}$
\begin{equation}
    \|\omega(\cdot,t)\|_{L^\infty} \geq
    \frac{\left|\omega^0(r_0,z_0)\right|}
    {1-\frac{t}{T_{max}}},
\end{equation}
and consequently for all $0\leq t<T_{max}$
\begin{equation} \label{IntegralBoundBKM}
    \int_0^t\|\omega(\cdot,\tau)\|_{L^\infty} \diff\tau 
    \geq
    \left|\omega^0(r_0,z_0)\right|T_{max}
    \log\left(\frac{T_{max}}{T_{max}-t}\right)
\end{equation}
\end{proposition}

\begin{proof}
Observe that for all 
$0\leq t<T_{max}$,
\begin{align}
    \|\omega(\cdot,t)\|_{L^\infty}
    &\geq 
    |\omega(r_0,z_0+t\phi^0(r_0,z_0),t)| \\
    &=
    \frac{|\omega^0(r_0,z_0)|}
    {1+t\partial_z\phi^0(r_0,z_0)} \\
    &=
    \frac{|\omega^0(r_0,z_0)|}
    {1-\frac{t}{T_{max}}}.
\end{align}
Integrating this lower bound immediately gives the lower bound \eqref{IntegralBoundBKM}, and this completes the proof.
\end{proof}

\begin{remark}
It can easily be seen from the formula for the evolution of the vorticity that singularities form at a rate $\sim \frac{1}{1-\frac{t}{T_{max}}}$, and that furthermore, the vorticity blows up in the $L^\infty$ norm at this rate whenever the initial vorticity is nonzero at a point where $\partial_z\phi^0$ is minimized. 
It is not clear, however, that a Beale-Kato-Majda criterion must hold in the case where the initial vorticity is zero at all points where $\partial_z\phi^0$ is minimized. In fact, there is a simple counterexample showing that the Beale-Kato-Majda criterion does not hold for all solutions of the infinite-dimensional vorticity equation that blowup in finite-time. Indeed, there are solutions of the infinite-dimensional vorticity equation that blowup in finite-time which are uniformly bounded up until $T_{max}$.
\end{remark}

\begin{proposition} \label{BKMcounterexample}
Take the initial data
\begin{equation}
    \omega^0(r,z)=2rz\exp(-r^2-z^2),
\end{equation}
and let $\omega \in 
C^\infty\left([0,T_{max});
C^\infty\left(\mathbb{R}^+
\times\mathbb{R}\right)\right)$ be the unique strong solution to the infinite-dimensional vorticity equation with this initial condition.
Then $T_{max}=1$, and for all $0\leq t<1$,
\begin{equation}
    |\omega(r,z,t)|\leq 1.
\end{equation}
\end{proposition}

\begin{proof}
We will again take $y=h(r,z,t)$ where $h$ is the back-to-labels map used in Theorems \ref{BurgersThm} and \ref{InfiniteVortThm}.
We know that 
\begin{equation}
    \omega(r,z,t)=
    \frac{\omega^0(r,y)}{1+t\partial_y\phi^0(r,y)}.
\end{equation}
It is simple to check that for the initial vorticity chosen, the corresponding initial potential is
\begin{equation}
    \phi^0(r,y)=-y\exp(-r^2-y^2),
\end{equation}
and so
\begin{equation}
    \partial_y\phi^0(r,y)=-(1-2y^2)\exp(-r^2-y^2).
\end{equation}
Clearly
\begin{equation}
    \inf_{\substack{r\in\mathbb{R}^+\\ y\in\mathbb{R}}}\partial_y\phi^0(r,y)
    =\partial_y\phi^0(0,0)=-1,
\end{equation}
is the sole minimum, and $\omega^0(0,0)=0$.
This implies that $T_{max}=1$ and that Proposition \ref{BKM} does not apply to this initial data.

In fact, it is straightforward to compute that the vorticity remains uniformly bounded.
We compute that for all $0\leq t<1$
\begin{align}
    |\omega(r,z,t)|
    &=
    \frac{|\omega^0(r,y)|}
    {1+t\partial_y\phi^0(r,y)} \\
    &=
    \frac{2r|y|\exp(-r^2-y^2)}
    {1-t(1-2y^2)\exp(-r^2-y^2)} \\
    &=
    \frac{2r|y|}{\exp(r^2+y^2)-t(1-2y^2)} \\
    &\leq
    \frac{2r|y|}{1+r^2+y^2-t(1-2y^2)} \\
    &=
    \frac{2r|y|}{1-t+r^2+(1+2t)y^2} \\
    &\leq 
    \frac{r^2+y^2}{r^2+y^2} \\
    &=
    1,
\end{align}
where we have used Young's inequality and the fact that $\exp(s)\geq 1+s$.
This completes the proof.
\end{proof}

\begin{remark}
    Note that for the example considered here $\omega^0\notin 
C_c^\infty\left(\mathbb{R}^+
\times\mathbb{R}\right)$,
but it is in the Schwarz class, and it is straightforward to check that this a strong enough regularity and decay condition to have all the wellposedness results considered above in this section.
\end{remark}

\begin{remark}
The separation of variables that occurs in the infinite dimensional vorticity equation, with the vorticity potential satisfying the one dimensional Burgers equation involving only the $z$-direction is a consequence of the structure of the divergence free constraint.
It is not possible to take the formal limit of divergence free constraint written as
\begin{equation}
    \nabla\cdot u=
    \partial_z u_z+\partial_r u_r+k\frac{u_r}{r}=0,
\end{equation}
however if we re-normalize this equation, writing the constraint as
\begin{equation}
    \frac{1}{k}\partial_z u_z+
    \frac{1}{k}\partial_r u_r
    +\frac{u_r}{r}=0,
\end{equation}
then it is very clear that taking the formal limit $k\to +\infty$ yields
\begin{equation}
    u_r=0.
\end{equation}
Recalling that $\omega=\partial_r u_z-\partial_z u_r$, the constraint $u_r=0$
then implies that
\begin{equation}
    u_z=\partial_r^{-1}\omega.
\end{equation}
Likewise, the divergence form of the vorticity equation (see \cite{MillerTsai} for details) can be expressed as
\begin{equation}
    \partial_t\omega
    +\partial_r (\omega u_r)
    +\partial_z(\omega u_z)
    =0.
\end{equation}
The constraint $u_r=0$ in the formal limit $d\to+\infty$ then yields
\begin{equation}
 \partial_t\omega
    +\partial_z(\omega u_z)
    =
    \partial_t\omega+u_z\partial_z\omega
    +\omega\partial_z u_z
    =0.
\end{equation}

This means that $\phi$ in the infinite dimensional vorticity equation is exactly $u_z$, and that axial compression is the dominant factor in singularity formation in very high dimension.
It is well known that axial compression/planar stretching is a dominant factor in singularity formation for the three dimensional Navier--Stokes equation, which can be expressed by regularity criteria in terms of $\lambda_2^+$, the positive part of the middle eigenvalue of the strain matrix
\cites{NeuPen}.
This result suggests that for very high dimensions axial compression/hyperplanar stretching may likely lead to singularity formation, and axial compression becomes stronger relative to hyperplanar stretching in very high dimensions.
This is because more directions to stretch out means that rate of hyperplanar stretching will go like a factor of $\frac{1}{d-1}$ the rate of axial compression due to the divergence free constraint, which requires that $\tr(S)=0$.
\end{remark}

\section{The full Euler equation as a perturbation of the model equation} \label{PerturbSection}

In the final section of this paper, 
we will show that axisymmetric, swirl-free solutions of the Euler equation in very high dimensions can be treated as a perturbation of the infinite-dimensional vorticity equation, where the size of the perturbation is $\epsilon=\frac{1}{d-2}$.
This is significant, because we can express solutions of the incompressible Euler equation as a perturbation of a model equation that is equivalent the 1D Burgers equation.
We will now prove \Cref{PerturbThmIntro}, which is restated for the reader's convenience.

\begin{theorem} \label{PerturbThm}
Suppose $u\in C\left([0,T_{max});H^s
\left(\mathbb{R}^d\right)\right)$ is an axisymmetric, swirl-free solution of the Euler equation in $d$-dimensional space, with $s>\frac{d}{2}+2, d\geq 3$.
Then the vorticity satisfies the equation
\begin{equation}
    \partial_t\omega 
    +\phi\partial_z\omega
    +\omega\partial_z\phi
    +\epsilon Q_\epsilon[\omega]=0,
\end{equation}
where
\begin{align}
    \phi&=\partial_r^{-1}\omega \\
    \epsilon&=\frac{1}{d-2},
\end{align}
and the operator $Q_\epsilon$ is defined by
\begin{equation}
    Q_\epsilon[\omega]=
    (-r\partial_z\phi
    +\epsilon\partial_z\sigma)
    \partial_r\omega
    +\left(\phi+r\omega
    -\frac{\sigma}{r}
    -\epsilon\partial_r\sigma
    \right)\partial_z\omega
    -\frac{\partial_z\sigma}{r}\omega.
\end{equation}
We define $\sigma$ to be the unique solution of the Poisson-type equation
\begin{equation}
    -\Delta_\epsilon \sigma
    =
    -\left(\partial_r^2+\partial_z^2\right)
    (r\phi),
\end{equation}
with decay as $r,|z|\to \infty$,
where
\begin{equation}
    -\Delta_\epsilon
    =
    -\epsilon\left(\partial_r^2
    +\partial_z^2\right)
    -\frac{1}{r}\partial_r
    +\frac{1}{r^2}.
\end{equation}
Note that $\sigma$ can be expressed as an integral against a Poisson-type kernel as in \eqref{PoissonKernel}.
\end{theorem}

\begin{proof}
We begin by recalling the formulation of the vorticity equation in terms of the re-normalized stream function at the beginning of this section,
\begin{align}
\partial_t\omega
    +u_r\partial_r\omega
    +u_z\partial_z\omega 
    -\frac{k}{r}u_r\omega&=0 \\
    u_r&= 
    \epsilon \partial_z \Tilde{\psi}\\
    u_z&=
    -\epsilon\partial_r\Tilde{\psi}
    -\frac{\Tilde{\psi}}{r} \\
    -\Delta_{\epsilon}\Tilde{\psi}
    &=\omega, \\
\end{align}
where we have made the substitution $\epsilon=\frac{1}{k}$.
We wish to invert the operator $-\Delta_{\epsilon}$ in the small epsilon limit.
We can see that if $\epsilon=0$, then
\begin{align}
    -\partial_r\left(
    \frac{\Tilde{\psi}}{r}\right)
    &=
    \left(-\frac{1}{r}\partial_r
    +\frac{1}{r^2}\right)\Tilde{\psi} \\
    &=
    \omega \\
    &=
    \partial_r\phi,
\end{align}
and so
\begin{equation}
    \Tilde{\psi}=-r\phi.
\end{equation}

We will write $\Tilde{\psi}$ as an $\epsilon$ perturbation of this function, with
\begin{equation}
    \Tilde{\psi}=-r\phi +\epsilon\sigma,
\end{equation}
where $\sigma$ is defined as above.
It is straightforward to verify that this stream function satisfies the equation $-\Delta_\epsilon \Tilde{\psi}=\omega$.
We can see that
\begin{align}
    -\Delta_\epsilon \Tilde{\psi}
    &=
    -\epsilon\Delta_\epsilon \sigma
    +\epsilon\left(\partial_r^2+\partial_z^2\right)
    (r\phi)
    +\partial_r \phi \\
    &=
    \omega.
\end{align}
Computing the velocity from the stream function then yields
\begin{align}
    u_r&=
    -\epsilon r\partial_z \phi
    +\epsilon^2 \partial_z \sigma \\
    u_z &=
    \epsilon r\partial_r \phi
    -\epsilon^2\partial_r\sigma
    -\epsilon\frac{\sigma}{r}
    +(1+\epsilon)\phi.
\end{align}
We can also see that
\begin{equation}
    k\frac{u_r}{r}
    =
    -\partial_z\phi
    +\epsilon\frac{\partial_z\sigma}{r}.
\end{equation}
Plugging this into the vorticity evolution equation, we find that
\begin{multline}
    \partial_t \omega
    + \left(-\epsilon r\partial_z \phi
    +\epsilon^2 \partial_z \phi\right)
    \partial_r\omega
    +\left(\epsilon r\partial_r \phi
    -\epsilon^2\partial_r\sigma
    -\epsilon\frac{\sigma}{r}
    +(1+\epsilon)\phi\right)
    \partial_z\omega \\
    -\left(-\partial_z \phi
    +\epsilon\frac{\partial_z\sigma}{r}\right)
    \omega=0.
\end{multline}
Rearranging terms, we find that 
\begin{multline}
    \partial_t\omega
    +\phi\partial_z\omega
    +\omega\partial_z\phi \\
    +\epsilon\left(
    \left(-r\partial_z \phi
    +\epsilon \partial_z \sigma\right)
    \partial_r\omega
    +\left( r\omega
    -\epsilon\partial_r\sigma
    -\frac{\sigma}{r}
    +\phi\right)\partial_z\omega 
    -\frac{\partial_z\sigma}{r}\omega
    \right)=0.
\end{multline}
This completes the proof.
\end{proof}

\begin{remark}
Note that although the size of the perturbation tends to zero as $d \to \infty$, there is no immediately natural way to lift finite-time blowup for the infinite-dimensional vorticity equation to blowup for the Euler equation in very high dimensions.
This is because the operator $Q_\epsilon$ is also becoming singular as $\epsilon \to 0$, due to the fact that $(-\Delta_\epsilon)^{-1}$ loses its smoothing effect as $\epsilon \to 0$.
If there is a way to show that the Euler equation blows up in finite-time based on a perturbative argument of the Burgers-type blowup of the infinite dimensional limit, then it will require very subtle estimates on this operator in the small $\epsilon$ limit. More likely, this is heuristic evidence of a phenomenon that can better be attacked by other methods.
\end{remark}

\section*{Acknowledgments}
The author was partially supported by the Natural Sciences and Engineering Research Council of Canada (NSERC) under the grants RGPIN-2023-04534 and by the Pacific Institute for the Mathematical Sciences as a PIMS postdoctoral fellow at the University of British Columbia. This PIMS postdoctoral fellowship was supported by NSERC under grant no. 568577-2022.

The author would like to thank the organizers of the conference ``Workshop for Research and Workforce Development in Fluid Mechanics I" at the University of Nebraska-Lincoln---Kazuo Yamazaki, Zachary Bradshaw, Aynur Bulut, Adam Larios, and Xukai Yan---for a wonderful conference and for the chance to contribute to this conference proceeding. Finally, the author would like to thank the referee for a very thorough reading of the manuscript and for helpful suggestions.

\bibliography{bib}

\end{document}